\documentclass[12pt, reqno]{amsart}
\usepackage{amsmath, amsthm, amscd, amsfonts, amssymb, graphicx, color}
\usepackage[bookmarksnumbered, colorlinks, plainpages]{hyperref}
\hypersetup{colorlinks=true,linkcolor=red, anchorcolor=green, citecolor=cyan, urlcolor=red, filecolor=magenta, pdftoolbar=true}

\newtheorem{theorem}{Theorem}[section]
\newtheorem{lemma}[theorem]{Lemma}
\newtheorem{proposition}[theorem]{Proposition}
\newtheorem{corollary}[theorem]{Corollary}
\theoremstyle{definition}

\theoremstyle{remark}
\newtheorem{remark}[theorem]{Remark}
\numberwithin{equation}{section}

\input{mathrsfs.sty}

\begin{document}
\title[Norm inequalities for elementary operators]{Unitarily invariant norm inequalities for elementary operators involving $G_{1}$ operators}
\author[F. Kittaneh]{Fuad Kittaneh}
\address{Department of Mathematics, The University of Jordan, Amman, Jordan}
\email{fkitt@ju.edu.jo}
\author[M.S. Moslehian]{Mohammad Sal Moslehian}
\address{Department of Pure Mathematics, Center Of Excellence in Analysis on Algebraic Structures (CEAAS), Ferdowsi University of Mashhad, P. O. Box 1159, Mashhad 91775, Iran}
\email{moslehian@um.ac.ir}
\author[M. Sababheh]{Mohammad Sababheh}
\address{Department of Basic Sciences, Princess Sumaya University for Technology, Amman, Jordan}
\email{sababheh@psut.edu.jo, sababheh@yahoo.com}


\begin{abstract}
In this paper, motivated by perturbation theory of operators, we present some upper bounds for $|||f(A)Xg(B)+ X|||$ in terms of $|||\,|AXB|+|X|\,|||$ and $|||f(A)Xg(B)- X|||$ in terms of $|||\,|AX|+|XB|\,|||$, where $A, B$ are $G_{1}$ operators, $|||\cdot|||$ is a unitarily invariant norm and $f, g$ are certain analytic functions. Further, we find some new upper bounds for the the Schatten $2$-norm of $f(A)X\pm Xg(B)$. Several special cases are discussed as well.
\end{abstract}

\maketitle

\section{Introduction}

Let ${\mathbb B}({\mathscr H})$ denote the $C^*$-algebra of all bounded linear operators on a separable complex Hilbert space ${\mathscr H}$ equipped with the operator norm $\|\cdot\|$. If $\dim \mathscr{H}=n$, we can identify $\mathbb{B}(\mathscr{H})$ with the matrix algebra $\mathbb{M}_n$ of all $n \times n$ matrices with entries in the complex field $\mathbb{C}$. If $z \in \mathbb{C}$, then we write $z$ instead of $zI$, where $I$ denotes the identity operator on $\mathscr{H}$. We write $A \geq 0$ when $A$ is positive (positive semi-definite for matrices). For any operator $A$ in the algebra $\mathbb{K}(\mathscr{H})$ of all compact operators, we denote by $\{s_j(A)\}$ the sequence of singular values of $A$, i.e. the eigenvalues $\lambda_j(|A|)$, where $|A|=(A^*A)^{1\over{2}}$, in decreasing order and repeated according to multiplicity. If the rank $A$ is $n$, we put $s_k(A)=0$ for any $k>n$.\\

In addition to the operator norm $\|\cdot\|$, which is defined on whole of ${\mathbb B}({\mathscr H})$, a unitarily invariant norm is a map $\left\vert \left\vert \left\vert \cdot \right\vert \right\vert \right\vert: \mathbb{K}(\mathscr{H}) \to [0,\infty]$ given by $\left\vert \left\vert \left\vert A \right\vert \right\vert \right\vert=g(s_1(A), s_2(A), \cdots)$, where $g$ is a symmetric norming function. The set $\mathcal{C}_{|||\cdot|||}=\{A \in \mathbb{K}(\mathscr{H}) : \left\vert \left\vert \left\vert A \right\vert \right\vert \right\vert < \infty \}$ is a closed self-adjoint ideal $\mathcal{J}$ of $\mathbb{B}(\mathscr{H})$ containing finite rank operators. It enjoys the properties:
\begin{itemize}
\item [(i)] For all $A,B\in\mathbb{B}(\mathscr{H})$ and $X \in \mathcal{J}$,
\begin{eqnarray}\label{u1}
\left\vert \left\vert \left\vert AXB\right\vert \right\vert \right\vert \leq
\left\vert \left\vert A\right\vert \right\vert \ \left\vert \left\vert
\left\vert X\right\vert \right\vert \right\vert \ \left\vert \left\vert
B\right\vert \right\vert\,.
\end{eqnarray}
\item[(ii)] If $X$ is a rank one operator, then
\begin{eqnarray}\label{u2}
\left\vert \left\vert \left\vert X\right\vert \right\vert \right\vert =\|X\|\,.
\end{eqnarray}
\end{itemize}
Inequality \eqref{u1} implies that $\left\vert \left\vert \left\vert UAV\right\vert \right\vert \right\vert =\left\vert \left\vert \left\vert A\right\vert \right\vert \right\vert $ for all unitary matrices $U,V\in \mathbb{B}(\mathscr{H})$ and all $A\in \mathcal{J}$. In addition, employing the polar decomposition of $X=W|X|$ with $W$ a partial isometry and \eqref{u1}, we have
\begin{eqnarray}\label{u3}
|||X|||=|||\ |X|\ |||\,.
\end{eqnarray}
The Ky Fan norms as an example of unitarily invariant norms are defined by $\| A\|
_{(k)}=\sum_{j=1}^{k}s_{j}(A)$ for $k=1,2,\ldots$. The Ky Fan
dominance theorem \cite[Theorme IV.2.2]{1} states that $\| A\|
_{(k)}\leq \| B\| _{(k)}\,\,(k=1,2,\ldots )$ if and only if
$|||A||| \leq |||B|||$ for all unitarily invariant norms
$|||\cdot|||$; see \cite{1, 6} for more information on unitarily invariant norms. For the sake of brevity, we will not explicitly mention this norm
ideal. Thus, when we consider $\left\vert \left\vert \left\vert A\right\vert \right\vert \right\vert $, we are assuming that $A$ belongs to the norm
ideal associated with $\left\vert \left\vert \left\vert \cdot \right\vert \right\vert \right\vert $. It is known that the Schatten $p$-norms $\|A\|_p=\left(\sum_{j=1}^\infty s_j^p(A)\right)^{1/p}$ are unitarily invariant for $1 \leq p < \infty$; cf. \cite[Section IV]{1}. We use the notation $A\oplus B$ for the diagonal block matrix ${\rm diag} (A,B)$.  Its singular values are $s_1(A), s_1(B), s_2(A), s_2(B), \cdots$. It is evident that
\begin{eqnarray}\label{plus}
\|A\oplus B\|=\max\{\|A\|,\|B\|\} \quad{\rm and}\quad \|A\oplus B\|_p=(\|A\|_p^p+\|B\|_p^p)^{1/p}\,.
\end{eqnarray}
The inequalities involving unitarily invariant norms have been of special interest; see e.g., \cite{mm}.

An operator $A\in\mathbb{B}(\mathscr{H})$ is called $G_{1}$ operator if the growth condition
\begin{eqnarray}\label{3}
\left\Vert (z-A)^{-1}\right\Vert =\frac{1}{{\rm{dist}}(z,\sigma (A))}
\end{eqnarray}
holds for all $z$ not in the spectrum $\sigma (A)$ of $A$. Here ${\rm{dist}}(z,\sigma (A))$ denotes the distance between $z$ and $\sigma
(A)$. It is known that hyponormal (in particular, normal) operators are $
G_{1}$ operators (see, e.g., \cite{9}).

Let $A\in\mathbb{B}(\mathscr{H})$ and let $f$ be a function which is analytic on an open neighborhood $
\Omega $ of $\sigma (A)$ in the complex plane. Then $f(A)$ denotes the
operator defined on $\mathbb{H}$ by the Riesz-Dunford integral as
\begin{eqnarray}
f(A)=\frac{1}{2\pi i}\int\limits_{C}f(z)(z-A)^{-1}dz,  \label{4}
\end{eqnarray}
where $C$ is a positively oriented simple closed rectifiable contour
surrounding $\sigma (A)$ in $\Omega $ (see, e.g., \cite[p. 568]{5}). The spectral mapping theorem asserts that $\sigma (f(A))=f(\sigma (A))$. Throughout this note,  $\mathbb{D}=\{z\in\mathbb{C}:\left\vert z\right\vert <1\}$ denotes the unit disk, $\partial\mathbb{D}$ stands for the boundary of $\mathbb{D}$ and $d_{A}={\rm{dist}}(\partial\mathbb{D},\sigma (A))$. In addition, we adopt the notation
$$\mathfrak{H}=\{f: \mathbb{D}\to \mathbb{C}: f \mbox{~is analytic}, \Re(f)>0 \mbox{~and} f(0)=1\}.$$

The Sylvester type equations $AXB\pm X=C$ have been investigated in matrix theory; see, e.g. \cite{BAO}. In addition, operators of the form $R(X)=\sum_{i=1}^nA_iXB_i$, in particular $\Delta_{A,B}=AXB- X$, are called elementary and have been studied in various aspects by several people; see, e.g. \cite{YAM}. Some mathematicians try to find some upper and lower bounds for norms of elementary operators; cf. \cite{SED}. Regarding $AXB+X$, it is shown in \cite{KIT2} that there is a constant $\gamma_p$ for any $1<p<\infty$ such that for any $A, B, X \in \mathbb{B}(\mathscr{H})$ such that $A, B$ are positive, it holds that $\|AXB+X\|_p\geq \gamma_p\|X\|_p$.

Several perturbation bounds for the norm of sum or difference of operators have been presented in the literature by employing some integral representations of certain functions; cf. \cite{BS, KIT}. For some other perturbation results the reader is referred to \cite{HK1, JK}. In this paper, we present some upper bounds for $|||f(A)Xg(B)\pm X|||$, where $A, B$ are $G_{1}$ operators, $|||\cdot|||$ is a unitarily invariant norm and $f, g\in \mathfrak{H}$. Further, we find some new upper bounds for the the Schatten $2$-norm of $f(A)X\pm Xg(B)$. Several applications are presented as well.


\section{Upper bounds for $|||f(A)Xg(B)\pm X|||$}

In this section, we find some upper bounds for $|||f(A)Xg(B)+ X|||$ in terms of $|||\,|AXB|+|X|\,|||$ and $|||f(A)Xg(B)- X|||$ in terms of $|||\,|AX|+|XB|\,|||$, where $A, B$ are $G_{1}$ operators, $|||\cdot|||$ is a unitarily invariant norm and $f, g\in \mathfrak{H}$, and present several consequences.\\
Our main result of this section reads as follows.

\begin{theorem}\label{T1}
If $A,B\in\mathbb{B}(\mathscr{H})$ are $G_{1}$ operators with $\sigma (A)\cup \sigma (B)\subset\mathbb{D}$ and $f, g \in \mathfrak{H}$, then for every $X\in\mathbb{B}(\mathscr{H})$ and for every unitarily invariant norm $\left\vert \left\vert \left\vert\cdot \right\vert \right\vert \right\vert $, the inequalities
\begin{eqnarray}
\left\vert \left\vert \left\vert f(A)Xg(B)+X\right\vert \right\vert
\right\vert \leq  \frac{2\sqrt{2}}{d_{A}d_{B}} \left\vert\left\vert \left\vert\,|AXB|+|X|\,\right\vert \right\vert \right\vert,  \label{5}
\end{eqnarray}
and
\begin{eqnarray}
\left\vert \left\vert \left\vert f(A)Xg(B)-X\right\vert \right\vert
\right\vert \leq  \frac{2\sqrt{2}}{d_{A}d_{B}} \left\vert\left\vert \left\vert\,|AX|+|XB|\,\right\vert \right\vert \right\vert,  \label{55}
\end{eqnarray}
hold.
\end{theorem}

\begin{proof}
We prove inequality \eqref{5}, the other inequality can be proved in a similar fashion.\\
It follows from the Herglotz representation theorem (see, e.g., \cite[p. 21]{4}) that $f\in \mathfrak{H}$ can be
represented as
\begin{eqnarray}
f(z)=\int\limits_{0}^{2\pi }\frac{e^{i\alpha}+z}{e^{i\alpha}-z}d\mu(\alpha)+i\Im f(0)=\int\limits_{0}^{2\pi }\frac{e^{i\alpha}+z}{e^{i\alpha}-z}d\mu(\alpha)  \label{6}
\end{eqnarray}
where $\mu $ is a positive Borel measure on the interval $[0,2\pi ]$ with
finite total mass $\int\limits_{0}^{2\pi }d\mu(\alpha)=f(0)=1$. Similarly $g(z)=\int\limits_{0}^{2\pi }\frac{e^{i\alpha}+z}{e^{i\alpha}-z}d\nu(\alpha)$ for some positive Borel measure $\nu$ on the interval $[0,2\pi ]$ with finite total mass $1$. We have
\begin{eqnarray*}
&&\hspace{-0.5in}f(A)Xg(B)+X\\
&=&\int\limits_{0}^{2\pi }\int\limits_{0}^{2\pi }\left[ \left( e^{i\alpha}-A\right)^{-1}
\left( e^{i\alpha}+A\right)X\left( e^{i\beta}+B\right) \left(
e^{i\beta}-B\right)^{-1}+X\right] d\mu(\alpha)d\nu(\beta).
\end{eqnarray*}
A simple computation shows that
\begin{align*}
&\hspace{-0.5cm}\left( e^{i\alpha}-A\right)^{-1}
\left( e^{i\alpha}+A\right)X\left( e^{i\beta}+B\right) \left(
e^{i\beta}-B\right)^{-1}+X\\
&=\left( e^{i\alpha}-A\right)^{-1} \left( e^{i\alpha}+A\right)X\left( e^{i\beta}+B\right)\left( e^{i\beta}-B\right)^{-1}\\
&\quad + \left(e^{i\alpha}-A\right)^{-1} \left( e^{i\alpha}-A\right)X\left(
e^{i\beta}-B\right) \left( e^{i\beta}-B\right)^{-1}\\
&=\left( e^{i\alpha}-A\right)^{-1}[\left( e^{i\alpha}+A\right)X\left( e^{i\beta}+B\right)+\left( e^{i\alpha}-A\right)X\left( e^{i\beta}-B\right)] \left( e^{i\beta}-B\right)^{-1}\\
&= 2\left( e^{i\alpha}-A\right)^{-1}(AXB+e^{i\alpha}Xe^{i\beta}) \left( e^{i\beta}-B\right)^{-1}.
\end{align*}
Thus, by inequality \eqref{u1}, we have
\begin{eqnarray} \label{s1}
&&\hspace{-1.5cm}\left\vert \left\vert \left\vert f(A)Xg(B)+X\right\vert \right\vert
\right\vert \nonumber\\
&\leq& \int\limits_{0}^{2\pi }\int\limits_{0}^{2\pi }2\left\Vert \left( e^{i\alpha
}-A\right)^{-1}\right\Vert \left\vert \left\vert \left\vert AXB+
e^{i\alpha}Xe^{i\beta}\right\vert \right\vert \right\vert \left\Vert \left( e^{i\alpha
}-B\right)^{-1}\right\Vert d\mu(\alpha)d\nu(\beta).\nonumber \\
\end{eqnarray}
Due to $A$ and $B$ are $G_{1}$ operators, we deduce from \eqref{3} that
\begin{eqnarray}\label{s2}
\left\vert \left\vert \left( e^{i\alpha}-A\right)^{-1}\right\vert
\right\vert =\frac{1}{{\rm{dist}}(e^{i\alpha},\sigma (A))}\leq \frac{1}{{\rm{dist}}(\partial\mathbb{D},\sigma (A))}=\frac{1}{d_{A}},
\end{eqnarray}
and similarly
\begin{eqnarray}\label{s3}
\left\vert \left\vert \left( e^{i\beta}-B\right)^{-1}\right\vert \right\vert \leq \frac{1}{d_{B}}.
\end{eqnarray}
In addition, we have
\begin{align}
&\hspace{-0.5cm}\left\vert\left\vert \left\vert  \left(AXB+
e^{i\alpha}Xe^{i\beta}\right) \oplus 0\right\vert \right\vert \right\vert \nonumber \\
&=\left\vert\left\vert \left\vert \left(e^{-i\beta}AXB+e^{i\alpha}X\right) \oplus 0 \right\vert \right\vert \right\vert \nonumber \\
&=\left\vert\left\vert \left\vert  \begin{bmatrix} e^{-i\beta}& e^{i\alpha}\\0&0\end{bmatrix}\begin{bmatrix} AXB& 0\\X&0\end{bmatrix} \right\vert \right\vert \right\vert\nonumber \\
&\leq \left\vert \left\vert \begin{bmatrix} e^{-i\beta}&  e^{i\alpha}\\ 0&0\end{bmatrix}  \right\vert \right\vert \,\left\vert\left\vert \left\vert\begin{bmatrix} AXB& 0\\X&0\end{bmatrix} \right\vert \right\vert \right\vert\qquad \qquad \qquad  \quad (\mbox{by inequality~} \eqref{u1})\nonumber \\
&= \sqrt{2} \left\vert\left\vert \left\vert \ \left|\begin{bmatrix} AXB& 0\\X&0\end{bmatrix}\right| \ \right\vert \right\vert \right\vert\nonumber  \\
&= \sqrt{2} \left\vert\left\vert \left\vert \,(|AXB|^2+|X|^2)^{1/2}\oplus 0 \, \right\vert \right\vert \right\vert\nonumber  \\
\label{msm0}&\leq \sqrt{2} \left\vert\left\vert \left\vert \,(|AXB|+|X|)\oplus 0 \, \right\vert \right\vert \right\vert.
\end{align}
To get the last inequality we applied a result of Ando and Zhan \cite[p. 775]{AZ} to the function $h(t)=t^{1/2}$. It states that for every positive operators $C, D$, every non-negative operator monotone function $h(t)$ on $[0,\infty)$ and every unitarily invariant norm $|||\cdot|||$  it holds that $|||h(A+B)||| \leq |||h(A)+h(B)|||$.\\
Now from the Ky Fan dominance theorem and \eqref{msm0} we infer that
\begin{eqnarray}\label{s4}
 \left\vert\left\vert \left\vert  AXB+ e^{i\alpha}Xe^{i\beta} \right\vert \right\vert \right\vert \leq \sqrt{2} \left\vert\left\vert \left\vert\, |AXB|+|X| \, \right\vert \right\vert \right\vert.
\end{eqnarray}
It follows from  inequalities \eqref{s1}, \eqref{s2}, \eqref{s3} and \eqref{s4} that
\begin{eqnarray*}
\left\vert \left\vert \left\vert f(A)Xg(B)+X\right\vert \right\vert
\right\vert &\leq&  \frac{2\sqrt{2}}{d_{A}d_{B}}\left\vert\left\vert \left\vert\, |AXB|+|X| \, \right\vert \right\vert \right\vert \int\limits_{0}^{2\pi }\int\limits_{0}^{2\pi }d\mu(\alpha)d\nu(\beta)\\
&=&\frac{2\sqrt{2}\mu \left( \partial\mathbb{D}\right)\nu \left( \partial\mathbb{D}\right) }{d_{A}d_{B}} \left\vert\left\vert \left\vert\, |AXB|+|X| \, \right\vert \right\vert \right\vert \nonumber\\
&\leq&  \frac{2\sqrt{2}}{d_{A}d_{B}} \left\vert\left\vert \left\vert\, |AXB|+|X| \, \right\vert \right\vert \right\vert
\end{eqnarray*}
as required.
\end{proof}


\begin{remark}\label{r1}
Under the assumptions of Theorem \ref{T1}, we have
\begin{eqnarray*}
\left\vert \left\vert \left\vert \left(f(A)Xg(B)+X\right)\oplus 0 \right\vert \right\vert
\right\vert \leq  \frac{4\sqrt{2}}{d_{A}d_{B}} \left\vert\left\vert \left\vert AXB\oplus X \right\vert \right\vert \right\vert\,.
\end{eqnarray*}
To see this, first note that, the Ky Fan dominance theorem and \eqref{5} yield that
\begin{eqnarray}\label{msm1}
\left\vert \left\vert \left\vert (f(A)Xg(B)+X) \oplus 0\right\vert \right\vert
\right\vert \leq  \frac{2\sqrt{2}}{d_{A}d_{B}} \left\vert\left\vert \left\vert\,(|AXB|+|X|) \oplus 0\,\right\vert \right\vert \right\vert\,.
\end{eqnarray}
On the other hand, by inequality (3) in \cite{HK2},  $s_j((C+D)/2) \leq s_j(C\oplus D)$ for operators $C$ and $D$. Hence $\left\vert\left\vert \left\vert (C+D) \oplus 0 \right\vert \right\vert \right\vert \leq 2 \left\vert\left\vert \left\vert C\oplus D \right\vert \right\vert \right\vert$. Utilizing \eqref{u3}, we therefore get
\begin{eqnarray*}
\left\vert\left\vert \left\vert\,(|AXB|+|X|) \oplus 0\,\right\vert \right\vert \right\vert \leq 2   \left\vert\left\vert \left\vert\,|AXB|\oplus |X|\,\right\vert \right\vert \right\vert=2\left\vert\left\vert \left\vert AXB\oplus X \right\vert \right\vert \right\vert\,
\end{eqnarray*}
from which and inequality \eqref{msm1}, we reach the required inequality.
\end{remark}

\begin{corollary}\label{c1}
Let $f, g\in \mathfrak{H}$ and $A\in\mathbb{B}(\mathscr{H})$ be a $G_{1}$ operator with $\sigma (A)\subset\mathbb{D}$. Then for every normal operator  $X\in\mathbb{B}(\mathscr{H})$ commuting with $A$ and for every unitarily invariant norm $\left\vert \left\vert \left\vert\cdot \right\vert \right\vert \right\vert $, the inequalities
\begin{eqnarray*}
\left\vert \left\vert \left\vert f(A)Xg(A^*)+X\right\vert \right\vert
\right\vert \leq  \frac{2}{d_{A}^2} \left\vert\left\vert \left\vert\ A|X|A^*+|X|\ \right\vert \right\vert \right\vert,
\end{eqnarray*}
and
\begin{eqnarray*}
\left\vert \left\vert \left\vert f(A)Xg(A^*)-X\right\vert \right\vert
\right\vert \leq  \frac{2}{d_{A}^2} \left\vert\left\vert \left\vert\ |AX|+|XA^*|\ \right\vert \right\vert \right\vert\,.
\end{eqnarray*}
are valid.
\end{corollary}
\begin{proof}
First, note that under the assumptions of normality of $X$ and normality of $AXB$ and using this fact that $||| C+D |||\leq |||\,|C|+|D|\,|||$ for any normal operators $C$ and $D$ (see \cite{BOU}), the constant $\sqrt{2}$ can be reduced to $1$ in \eqref{s4}.\\
Second, recall that the Fuglede--Putnam theorem states that if $A\in {\mathbb B}({\mathscr H})$ is an operator, $X\in {\mathbb B}({\mathscr H})$ is normal and $AX=XA$, then $AX^*=X^*A$; see \cite{MN} and references therein. Thus if $X$ is a normal operator commuting with a $G_{1}$ operator $A$, then $AXA^*$ is normal, $|AXA^*|=A|X|A^*$ and $A^*$ is a $G_1$ operator with  $d_{A^*}=d_A$. Hence we get the required inequalities by employing Theorem \ref{T1}.
\end{proof}
Next, letting $A=B$ in \eqref{55} of Theorem \ref{T1}, we get the following inequality.
\begin{corollary}\label{C2}
Let $f, g\in \mathfrak{H}$ and $A\in\mathbb{B}(\mathscr{H})$ be a $G_{1}$ operator with $\sigma (A)\subset\mathbb{D}$. Then \begin{eqnarray*}
\left\vert \left\vert \left\vert f(A)Xg(A)-X\right\vert \right\vert
\right\vert \leq  \frac{2\sqrt{2}}{d_{A}^2} \left\vert\left\vert \left\vert \,|AX|+|XA|\,\right\vert \right\vert \right\vert
\end{eqnarray*}
for every $X\in\mathbb{B}(\mathscr{H})$ and for every unitarily invariant norm $\left\vert \left\vert \left\vert
\cdot \right\vert \right\vert \right\vert $.
\end{corollary}

Setting $X=I$ in Theorem \ref{T1}, we obtain the following result.

\begin{corollary}\label{C3}
Let $f, g\in \mathfrak{H}$ and $A,B\in\mathbb{M}_n$ be $G_{1}$ matrices such that $\sigma (A)\cup \sigma (B)\subset\mathbb{D}$. Then for every unitarily invariant norm $\left\vert \left\vert \left\vert
\cdot \right\vert \right\vert \right\vert $,
\begin{eqnarray*}
\left\vert \left\vert \left\vert f(A)g(B)+I\right\vert \right\vert
\right\vert \leq  \frac{2\sqrt{2}}{d_{A}d_{B}} \left\vert\left\vert \left\vert\,|AB|+I\,\right\vert \right\vert \right\vert
\end{eqnarray*}
and
\begin{eqnarray*}
\left\vert \left\vert \left\vert f(A)g(B)-I\right\vert \right\vert
\right\vert \leq  \frac{2\sqrt{2}}{d_{A}d_{B}} \left\vert\left\vert \left\vert\,|A|+|B|\,\right\vert \right\vert \right\vert .
\end{eqnarray*}
\end{corollary}
To achieve our next result, we need the following lemma. Its proof is standard but we provide a proof for the sake of completeness.
\begin{lemma}\label{lem1}
If $A\in \mathbb{B}(\mathscr{H})$ is self-adjoint and $f$ is a continuous complex function on $\sigma(A)$, then $f(UAU^*)=Uf(A)U^*$ for all unitaries $U$.
\end{lemma}
\begin{proof}
By the Stone-Weierstrass theorem, there is a sequence $(p_n)$ of polynomials uniformly converging to $f$ on $\sigma(A)$. Hence
$$f(UAU^*)=\lim_np_n(UAU^*)=U(\lim_np_n(A))U^*=Uf(A)U^*\,.$$
Note that $\sigma(UAU^*)=\sigma(A)$.
\end{proof}

\begin{proposition}\label{C4}
Let $f, g\in \mathfrak{H}$ and $A\in\mathbb{B}(\mathscr{H})$ be a positive operator with $\sigma (A)\subset [0,1)$. Then for every unitarily invariant norm $\left\vert \left\vert \left\vert
\cdot \right\vert \right\vert \right\vert $ and every unitary operator $U\in \mathbb{B}(\mathscr{H})$, it holds that
\begin{eqnarray*}
\left\vert \left\vert \left\vert f(A)Ug(A)-U\right\vert \right\vert
\right\vert \leq  \frac{2\sqrt{2}}{d_{A}^2} \left\vert\left\vert \left\vert\,AU+UA\,\right\vert \right\vert \right\vert .
\end{eqnarray*}
\end{proposition}
\begin{proof}
\begin{eqnarray*}
\left\vert \left\vert \left\vert f(A)Ug(A)-U\right\vert \right\vert\right\vert &=& \left\vert \left\vert \left\vert f(A)I(Ug(A)U^*)-I \right\vert \right\vert\right\vert\\
&=& \left\vert \left\vert \left\vert f(A)I g(UAU^*)-I\right\vert \right\vert\right\vert \qquad\qquad\qquad\quad (\mbox{by Lemma~} \ref{lem1})\\
&\leq& \frac{2\sqrt{2}}{d_{A}d_{UAU^*}} \left\vert \left\vert \left\vert\, |AI|+|IUAU^*|\, \right\vert \right\vert\right\vert\qquad (\mbox{by inequality~} \eqref{55})\\
&=& \frac{2\sqrt{2}}{d_{A}^2} \left\vert \left\vert \left\vert AU+UA \right\vert \right\vert\right\vert\,,
\end{eqnarray*}
since $d_{UAU^*}={\rm{dist}}(\partial\mathbb{D},\sigma (UAU^*))={\rm{dist}}(\partial\mathbb{D},\sigma (A))=d_A$.
\end{proof}

\section{Upper bounds for $|||f(A)X\pm Xg(B)|||$}

In the first result of this section, we find some upper bounds for $|||f(A)X\pm Xg(B)|||$.

\begin{theorem}\label{T1_N}
Let $A,B\in\mathbb{B}(\mathscr{H})$ be $G_{1}$ operators such that $\sigma (A)\cup \sigma (B)\subset\mathbb{D}$ and $X\in\mathbb{B}(\mathscr{H})$. Let $f, g \in \mathfrak{H}$ and $\left\vert \left\vert \left\vert\cdot \right\vert \right\vert \right\vert $ be a unitarily invariant norm. Then
$$\vert\vert\vert f(A)X+Xg(B)\vert\vert\vert \leq \frac{2\sqrt{2}}{d_Ad_B}\vert\vert\vert \;|AXB|+|X|\;\vert\vert\vert $$ and
$$\vert\vert\vert f(A)X-Xg(B)\vert\vert\vert \leq \frac{2\sqrt{2}}{d_Ad_B}\vert\vert\vert \;|AX|+|XB|\;\vert\vert\vert \,.$$
\end{theorem}
\begin{proof}
Noting that
$$\int_{0}^{2\pi}d\mu(\alpha)=\int_{0}^{2\pi}d\nu(\beta)=1,$$ we have

\begin{align*}
&f(A)X+Xg(B)\\
&=\int_{0}^{2\pi}(e^{i\alpha}-A)^{-1}(e^{i\alpha}+A)X\;d\mu(\alpha)+\int_{0}^{2\pi}X(e^{i\beta}+B)(e^{i\beta}-B)^{-1}\;d\nu(\beta)\\
&=\int_{0}^{2\pi}\int_{0}^{2\pi}\left\{(e^{i\alpha}-A)^{-1}(e^{i\alpha}+A)X(e^{i\beta}-B)(e^{i\beta}-B)^{-1}+\right.\\
&\quad \left.+(e^{i\alpha}-A)^{-1}(e^{i\alpha}-A)X(e^{i\beta}+B)(e^{i\beta}-B)^{-1}\right\}d\mu(\alpha)d\nu(\beta)\\
&=\int_{0}^{2\pi}\int_{0}^{2\pi}(e^{i\alpha}-A)^{-1}\left[(e^{i\alpha}+A)X(e^{i\beta}-B)\right.\\
&\quad\left. +(e^{i\alpha}-A)X(e^{i\beta}+B)\right](e^{i\beta}-B)^{-1}d\mu(\alpha)d\nu(\beta)\\
&=2\int_0^{2\pi}\int_{0}^{2\pi}(e^{i\alpha}-A)^{-1}(e^{i\alpha}Xe^{i\beta}-AXB)(e^{i\beta}-B)^{-1}d\mu(\alpha)d\nu(\beta).
\end{align*}
Applying the same reasoning as in the proof of Theorem \ref{T1}, we get the required inequality. The proof of the second inequality can be completed similarly.
\end{proof}
The next result reads as follows.
\begin{proposition}\label{T2_N}
Under the same assumptions of Theorem \ref{T1_N}, the inequalities
\begin{align*}
\vert\vert\vert f(A)X&+f(A)Xg(B)+Xg(B)\vert\vert\vert \\
&\leq \frac{\sqrt{2}}{d_Ad_B}\left(\vert\vert\vert \;|AXB|+|X|\;\vert\vert\vert +\vert\vert\vert \;|XB|+|X|\;\vert\vert\vert +\vert\vert\vert \;|AX|+|X|\;\vert\vert\vert
\right),
\end{align*}
and
\begin{align*}
\vert\vert\vert f(A)X&+f(A)Xg(B)+Xg(B)\vert\vert\vert \\
&\leq \frac{2}{d_Ad_B}\left(\vert\vert\vert \;|AXB|+|AX|+|XB|+3|X|\;\vert\vert\vert
\right)
\end{align*}
hold.
\end{proposition}
\begin{proof}
We prove the first inequality. As before, we have
\begin{align}\label{needed_1_thm_2}
f(A)X+Xg(B)&=2\int_0^{2\pi}\int_{0}^{2\pi}(e^{i\alpha}-A)^{-1}(e^{i\alpha}Xe^{i\beta}-AXB)(e^{i\beta}-B)^{-1}d\mu(\alpha)d\nu(\beta)
\end{align}
and
\begin{align}\label{needed_2_thm_2}
f(A)Xg(B)&=\int_{0}^{2\pi}\int_{0}^{2\pi}(e^{i\alpha}-A)^{-1}(e^{i\alpha}+A)X(e^{i\beta}+B)(e^{i\beta}-B)^{-1}d\mu(\alpha)d\nu(\beta).
\end{align}
Adding \eqref{needed_1_thm_2} and \eqref{needed_2_thm_2}, we get
\begin{align}
\nonumber f(A)X&+f(A)Xg(B)+Xg(B)\\
\nonumber&=\int_{0}^{2\pi}\int_{0}^{2\pi}(e^{i\alpha}-A)^{-1}
\left[2(e^{i\alpha}Xe^{i\beta}-AXB)\right.\\
\nonumber& \quad \left.+(e^{i\alpha}+A)X(e^{i\beta}+B)\right](e^{i\beta}-B)^{-1}d\mu(\alpha)d\nu(\beta)\\
\nonumber&=\int_{0}^{2\pi}\int_{0}^{2\pi}(e^{i\alpha}-A)^{-1}
\left[(e^{i\alpha}Xe^{i\beta}-AXB)+(e^{i\alpha}Xe^{i\beta}+e^{i\alpha}XB)\right.\\
\nonumber&\quad \left.+(e^{i\alpha}Xe^{i\beta}+AXe^{i\beta})\right](e^{i\beta}-B)^{-1}d\mu(\alpha)d\nu(\beta).
\end{align}
Consequently,
\begin{align}\label{needed_3_thm_2}
&\nonumber \vert\vert\vert f(A)X+f(A)Xg(B)+Xg(B)\vert\vert\vert \\
&\leq\frac{1}{d_Ad_B}\left(\vert\vert\vert e^{i\alpha}Xe^{i\beta}-AXB\vert\vert\vert +\vert\vert\vert e^{i\alpha}Xe^{i\beta}+e^{i\alpha}XB\vert\vert\vert +\vert\vert\vert e^{i\alpha}Xe^{i\beta}+AXe^{i\beta}\vert\vert\vert \right).
\end{align}
As before, it can be shown that
\begin{align}\label{needed_4_thm_2}
\vert\vert\vert e^{i\alpha}Xe^{i\beta}-AXB\vert\vert\vert \leq \sqrt{2}\vert\vert\vert \;|AXB|+|X|\;\vert\vert\vert .
\end{align}
Moreover,
\begin{align}\label{needed_5_thm_2}
\nonumber \vert\vert\vert e^{i\alpha}Xe^{i\beta}+e^{i\alpha}XB\vert\vert\vert &=\vert\vert\vert Xe^{i\beta}+XB\vert\vert\vert \\
\nonumber&=\vert\vert\vert e^{0i}Xe^{i\beta}+IXB\vert\vert\vert\\
&\leq \sqrt{2}\vert\vert\vert \;|XB|+|X|\;\vert\vert\vert .
\end{align}
Similarly,
\begin{align}\label{needed_6_thm_2}
\vert\vert\vert e^{i\alpha}Xe^{i\beta}+AXe^{i\beta}\vert\vert\vert \leq\sqrt{2}\vert\vert\vert \;|AX|+|X|\;\vert\vert\vert .
\end{align}
Now considering \eqref{needed_4_thm_2}, \eqref{needed_5_thm_2} and \eqref{needed_6_thm_2} in \eqref{needed_3_thm_2}, we get the desired inequality. The proof of the second inequality can be completed by an argument similar to that used in the proof of \eqref{msm0}.
\end{proof}

We conclude this article by presenting some inequalities involving the Hilbert-Schmidt norm $\|\cdot\|_2.$
\begin{theorem}\label{hilb}
Let $A,B\in\mathbb{M}_n$ be Hermitian matrices satisfying $\sigma(A)\cup \sigma(B)\subset \mathbb{D}$ and let $f, g\in \mathfrak{H}$. Then
\begin{align*}
\|f(A)X\pm Xg(B)\|_2\leq \left\|\frac{X+|A|X}{d_A}+\frac{X+X|B|}{d_B}\right\|_2,
\end{align*}
and
\begin{align*}
\|f(A)Xg(B)\pm X\|_2\leq \left\|\frac{I+|A|}{d_A}X\frac{I+|B|}{d_B}+X\right\|_2.
\end{align*}
\end{theorem}
\begin{proof}
We prove the first inequality. The second inequality follows similarly.
Let $A=UD(\nu_j)U^*$ and $B=VD(\mu_k)V^*$ be the spectral decomposition of $A$ and $B$ and let $Y=U^*XV:=[y_{jk}].$ Noting that $|e^{i\alpha}-\lambda_j|\geq d_A$ and $|e^{i\beta}-\mu_k|\geq d_B,$ we have
\begin{align*}
\|f(A)X\pm Xg(B)\|_2^2&=\sum_{j,k}|f(\lambda_j)\pm g(\mu_k)|^2|y_{jk}|^2\\
&=\sum_{j,k}\left|\int_{0}^{2\pi}\frac{e^{i\alpha}+\lambda_j}{e^{i\alpha}-\lambda_j}d\mu(\alpha)\pm
\int_{0}^{2\pi}\frac{e^{i\beta}+\mu_k}{e^{i\beta}-\mu_k}d\nu(\beta)\right|^2|y_{jk}|^2\\
&\leq \sum_{j,k}\left(\int_{0}^{2\pi}\frac{|e^{i\alpha}+\lambda_j|}{|e^{i\alpha}-\lambda_j|}d\mu(\alpha)+
\int_{0}^{2\pi}\frac{|e^{i\beta}+\mu_k|}{|e^{i\beta}-\mu_k|}d\nu(\beta)\right)^2|y_{jk}|^2\\
&\leq\sum_{j,k}\left(\frac{1+|\lambda_j|}{d_A}+\frac{1+|\mu_k|}{d_B}\right)^2|y_{jk}|^2\\
&=\left\|\frac{X+|A|X}{d_A}+\frac{X+X|B|}{d_B}\right\|_2^2,
\end{align*}
which completes the proof.
\end{proof}

\begin{remark}
The results in Theorem \ref{hilb} can be extended to infinite dimensional separable complex Hilbert spaces. This can be done by using a result in \cite{voi}, which insures that if $A,B$ are normal operators acting on an infinite dimensional separable Hilbert space, then $A$ and $B$ are Hilbert-Schmidt perturbations of diagonal operators. That is, given $\epsilon>0,$ there exist diagonal operators $\Gamma_1,\Gamma_2$ and unitary operators $U,V$ such that $\|A-U\Gamma_1U^*\|_2<\epsilon$ and $\|B-V\Gamma_2V^*\|_2<\epsilon$.
\end{remark}


\end{document}